\documentclass[11pt,a4paper,oneside]{amsart}
\usepackage[english]{babel}
\usepackage[utf8]{inputenc}
\usepackage{amssymb}
\usepackage{microtype}

\usepackage{extarrows}

\newcommand{\longtwoheadrightarrow}{\ensuremath{\relbar\joinrel\twoheadrightarrow}}

\usepackage[bottom=5cm]{geometry}

\usepackage[bookmarks=false]{hyperref}
\hypersetup{
colorlinks=true,	
linkcolor=black,	
citecolor=black,	
filecolor=black,	
urlcolor=black		
}

\DeclareMathOperator{\lcm}{lcm}
\newcommand{\norm}[1]{\left\lVert#1\right\rVert}

\DeclareMathOperator{\Hom}{Hom}
\DeclareMathOperator{\ord}{ord}
\DeclareMathOperator{\Stab}{Stab}
\DeclareMathOperator{\Ker}{Ker}
\DeclareMathOperator{\codim}{codim}
\DeclareMathOperator{\Gal}{Gal}
\DeclareMathOperator{\Id}{Id}
\DeclareMathOperator{\card}{card}
\DeclareMathOperator{\GL}{GL}

\DeclareMathOperator{\supp}{supp}
\DeclareMathOperator{\conv}{conv}
\DeclareMathOperator{\vol}{vol}
\DeclareMathOperator{\diam}{diam}

\newcommand{\Gm}{\mathbb{G}_{\textrm{m}}}
\newcommand{\Qab}{\mathbb{Q}^{\textrm{ab}}}

\newcommand{\Vtors}{\overline{V_{\textrm{tors}}}}
\newcommand{\Wtors}{\overline{\Wtorss}}
\newcommand{\Vtorss}{V_{\textrm{tors}}}
\newcommand{\Wtorss}{W_{\textrm{tors}}}

\newcommand{\bfeta}{\boldsymbol{\eta}}
\newcommand{\bftau}{\boldsymbol{\tau}}
\newcommand{\bfxi}{\boldsymbol{\xi}}
\newcommand{\bfomega}{\boldsymbol{\omega}}
\newcommand{\bflambda}{\boldsymbol{\lambda}}
\newcommand{\bfa}{\boldsymbol{a}}

\newcommand{\bfe}{\boldsymbol{e}}

\newcommand{\bfg}{\boldsymbol{g}}
\newcommand{\bft}{\boldsymbol{t}}
\newcommand{\bfv}{\boldsymbol{v}}
\newcommand{\bfx}{\boldsymbol{x}}
\newcommand{\bfy}{\boldsymbol{y}}

\theoremstyle{plain}
\newtheorem{theorem}{Theorem}[section]
\newtheorem{proposition}[theorem]{Proposition}
\newtheorem{corollary}[theorem]{Corollary}
\newtheorem{lemma}[theorem]{Lemma}
\newtheorem{conjecture}[theorem]{Conjecture}
\newtheorem*{proposition*}{Proposition}
\newtheorem*{theorem13}{Theorem 1.3}
\newtheorem*{theorem14}{Theorem 1.4}
\theoremstyle{definition}

\newtheorem*{example}{Example}

\newtheorem*{remark}{Remark}
\newtheorem{Remark}[theorem]{Remark}

\begin{document}

\title{The number of maximal torsion cosets in subvarieties of tori}

\author{César Martínez}
\address{Laboratoire de math\'ematiques Nicolas Oresme, CNRS UMR 6139, Universit\'e de Caen. BP 5186, 14032 Caen Cedex, France}
\email{cesar.martinez-metzmeier@unicaen.fr}

\date{\today}

\thanks{This research was partially financed by the CNRS project PICS 6381 ``G\'eom\'etrie diophantienne et calcul formel", and the Spanish project MINECO MTM2012-38122-C03-02.}

\begin{abstract}
We present sharp bounds on the number of maximal torsion cosets in a subvariety of the complex algebraic torus $\Gm^n$.
Our first main result gives a bound in terms of the degree of the defining polynomials.
A second result gives a bound in terms of the toric degree of the subvariety.

As a consequence, we prove the conjectures of Ruppert and of Aliev and Smyth on the number of isolated torsion points of a hypersurface. These conjectures bound this number in terms of the multidegree and the volume of the Newton polytope of a polynomial defining the hypersurface, respectively.
\end{abstract}

\maketitle



\section{Introduction}

Let $\Gm^n=(\mathbb{C}^\times)^n$ be the complex algebraic torus of dimension~$n$.
A torsion point of~$\Gm^n$ is an $n$-tuple of roots of unity.
Given $V$ a subvariety of $\Gm^n$, we call $\Vtorss$ the set of torsion points contained in $V$ and we denote by $\Vtors$ its Zariski closure.

The toric Manin-Mumford conjecture states that
$\Vtors$ is a finite union of torsion cosets, that is translates by torsion points of algebraic subtori of $\Gm^n$.
This was proved by Ihara, Serre and Tate for $\dim (V)=1$ \cite[Theorem 6.1]{Lan83} and by Laurent for higher dimensions \cite[Théorème 2]{Lau84}.

In this article, we focus on finding a sharp upper bound for the number of maximal torsion cosets in $V$ and their degrees.
It was already proved by Laurent \cite{Lau84}
that, if $V$ is defined over a number field $\mathbb{K}$ by a set of polynomials of degree at most $\delta$
and height at most $\eta$,
the number of maximal torsion cosets in $V$ and their degree
is effectively bounded in terms of $n$, $\delta$, $\eta$ and $[\mathbb{K}:\mathbb{Q}]$.
Later, Bombieri and Zannier \cite{BZ95}, following the work of Zhang \cite{Zha95}, showed that both the number of maximal torsion cosets and the degree of their defining polynomials can be bounded just in terms of $n$ and $\delta$.

Furthermore, Schmidt \cite{Sch96} obtained an explicit upper bound for the number of maximal torsion cosets in $V$.
Combined with a result of Evertse \cite{Eve99}, he bounds the number of maximal torsion cosets by
\[
(11\delta)^{n^2}
\left(\begin{matrix}
n + \delta\\
\delta
\end{matrix}\right)^{3
\left(\begin{smallmatrix}
n + \delta\\
\delta
\end{smallmatrix}\right)^2}
\mbox{.}
\]

Using a different approach, Ruppert \cite{Rup93} presented an algorithm to determine the torsion cosets of a variety $V\subset\Gm^n$.
Ruppert's approach treats first the case $\dim(V)=1$
where, given $(d_1,\ldots,d_n)$ the multidegree of $V\subset (\mathbb{P}^1)^n$,
he is able to bound the number of isolated torsion points in $V$ by $22\max (d_i)\min (d_i)$.
Afterwards, he extends his algorithm to some specific varieties in higher dimension,
which leads him to formulate the following conjecture:

\begin{conjecture}[Ruppert]\label{ConjR}
Let $f\in\mathbb{C}[X_1,\ldots,X_n]$ of multidegree $(d_1,\ldots,d_n)$.
The number of isolated torsion points on $Z(f)\subset\Gm^n$ is bounded by $c_n d_1\cdots d_n$,
where $c_n$ is a constant depending only on $n$.
\end{conjecture}

Beukers and Smyth \cite{BS02} reconsidered the problem for $n=2$ from a similar point of view to Ruppert's,
being able to refine the bound in terms of the Newton polytope.
Given $f\in\mathbb{C}[X,Y]$ a polynomial,
they bound the number of torsion points of~$Z(f)$ by~$22\textrm{vol}_2(\Delta)$,
where $\Delta$ denotes the Newton polytope of $f$.
This leads Aliev and Smyth~\cite{AS12} to strengthen the original conjecture of Ruppert as follows:

\begin{conjecture}[Aliev-Smyth]\label{ConjAS}
Let $f\in\mathbb{C}[X_1,\ldots,X_n]$,
the number of isolated torsion points on $Z(f)\subset\Gm^n$
is bounded by $c_n\vol_n(\Delta)$,
where $c_n$ is a constant depending only on $n$ and $\Delta$ is the Newton polytope of $f$.
\end{conjecture}

For a general polynomial $f\in\mathbb{C}[X_1,\ldots,X_n]$ of degree $\delta$, these conjectures imply that the number of isolated torsion points on $Z(f)$ is bounded by
\begin{equation}\label{EQdensecase}
c_n \delta^n\mbox{.}
\end{equation}

Aliev and Smyth \cite{AS12} extended Beukers and Smyth's algorithm to higher dimensions and obtained a bound,
far from the conjectured one.
For $f\in\mathbb{C}[X_1,\ldots,X_n]$ of degree $d$, they bound the number of maximal torsion cosets in $V$ by
\begin{equation}\label{EQAS}
\displaystyle
c_1(n)
d^{c_2(n)}
\mbox{,}
\end{equation}
where $c_1(n)=n^{\frac{3}{2}(2+n)5^n}$ and $c_2(n)=\frac{1}{16}(49\cdot d^{n-2}-4n-9)$.

For sparse representation of polynomials,
Leroux \cite{Ler12} obtained an algorithm to compute the maximal torsion cosets in $V$.
As a consequence, he is able to bound the number of maximal torsion cosets in $V$
in terms of the number of nonzero coefficients of the defining polynomials of $V$.
For dense polynomials the bound has similar order to~(\ref{EQAS}). 

A much better bound follow as a particular case of the study of points of small height on subvarieties of tori by Amoroso and Viada \cite[Corollary 5.4]{AV09}. 
Let $V$ be a subvariety of $\Gm^n$ of codimension $k$ defined by polynomials of degree at most $\delta$,
and let $\Vtorss^j$ be the $j$-equidimensional part of $\Vtors$. They obtain the following bound:
\[
\deg(\Vtorss^j)\leq \big(\delta (200n^5 \log(n^2\delta))^{(n-k)n(n-1)}\big)^{n-j}
\mbox{.}
\]
Note that $\delta$ can be taken by the degree of $V$ in the case that $V$ is a hypersurface
(as in the statements of the conjectures).
Thus, we observe that the number of isolated torsion points of $V$ gives (\ref{EQdensecase}), up to a logarithmic factor.

\mbox{}

In this article we combine the approach of Ruppert and Aliev and Smyth with the methods of Amoroso and Viada  to prove both, Ruppert's and Aliev and Smyth's conjectures. Our first main result is the following:
\begin{theorem}\label{teo}
Let $V\subset\Gm^n$ be a variety of dimension $d$ defined by polynomials of degree at most $\delta$.
Then
\[
\deg(\Vtorss^j)\leq c_n \delta^{n-j}
\]
for every $j=0,\ldots,d$\,, where
$c_n=((2n-1)(n-1)(2^{2n}+2^{n+1}-2))^{nd}$.
\end{theorem}

Applied to a general hypersurface of degree $\delta$, this proves the bound in (\ref{EQdensecase}).

\mbox{}

Let $\Delta\subset \mathbb{R}^n$ be a convex polytope with integer vertex and let $W\subset\Gm^n$ be a variety of dimension $d$. We define the toric degree as
\[
\deg_{\Delta}(W)=\card(W\cap Z)
\mbox{,}
\]
where $Z$ is a variety of codimension $d$ given by $d$ general polynomials $f_1,\ldots,f_d$ with Newton poytope $\Delta$.
Using John's theorem \cite[Theorem III]{Joh48}, we are able to translate this result to prove the conjectures.
Our second main result is the following:
\begin{theorem}\label{TEO}
Let $V\subset\Gm^n$ be a variety of dimension $d$ and defined by polynomials with support in the convex polytope with integer vertex $\Delta$.
Then
\[
\deg_{\Delta}(\Vtorss^j)\leq \widetilde{c}_{n,j}\vol_n(\Delta)
\]
for every $j=0,\ldots,d$\,, where
\[
\widetilde{c}_{n,j}=((2n-1)(n-1)(2^{2n}+2^{n+1}-2))^{(n-1)(n-j)} 2^n n^{2n}\omega_n^{-1}
\mbox{,}
\]
with $\omega_n$ representing the volume of the $n$-sphere.
\end{theorem}

Note that $\deg_\Delta(\Vtorss^0)=\deg(\Vtorss^0)$ and so we obtain the following result as a particular case to Theorem \ref{TEO}.
\begin{corollary}\label{cor}
Let $f\in\overline{\mathbb{Q}}[X_1,\ldots,X_n]$ and let $\Delta\subset\mathbb{R}^n$ be a convex body such that $\supp(f)\subset\Delta$.
Then the number of isolated torsion points on the hypersurface $Z(f)\subset\Gm^n$
is bounded by
\[
\widetilde{c}_n\vol_n(\Delta)\mbox{,}
\]
where $\widetilde{c}_n=((2n-1)(n-1)(2^{2n}+2^{n+1}-2))^{n(n-1)} 2^n n^{2n}\omega_n^{-1}$.
\end{corollary}

Given $f$ a polynomial of multidegree $(d_1,\ldots,d_n)$, we can take $\Delta=[0,d_1]\times\cdots\times [0,d_n]$ which proves Ruppert's conjecture (Conjecture \ref{ConjR}). Moreover, taking $\Delta$ as the Newton polytope of $f$ proves Aliev and Smyth's conjecture (Conjecture \ref{ConjAS}).

\vspace*{.5cm}

To discard the logarithmic error term in \cite[Theorem 1.2]{AV09},
we reformulate the main theorems of Amoroso and Viada so that they suit our particular case of torsion subvarieties.

To do so, first we extend the argument introduced by Beukers and Smyth in \cite{BS02} to a more algebraic setting.
In Proposition \ref{proposition21} we get, for any irreducible subvariety $V$ of $\Gm^n$,
another variety $V'\subset\Gm^n$ with the same dimension and similar degree,
such that $\Vtors$ lies in the intersection $V\cap V'\subsetneq V$.
Moreover, this $V'$ can be obtained explicitly from our initial $V$.

Next, in Theorem \ref{t21}, we use the Hilbert function to consider, instead of the subvariety~$V'$,
a hypersurface $Z$ satisfying $\Vtors\subset V\cap Z\subsetneq V$.
To do that we rely on both an upper and a lower bound for the Hilbert function,
the upper bound being a result of Chardin \cite{Cha88} and
the lower bound a result of Chardin and Philippon \cite{CP99}.
By using this bounds, we obtain Lemma \ref{lemmaHilbert}, which serves as a bridge between $V'$ and $Z$ and is, therefore, the key element in our proof of Theorem \ref{t21}.

Afterwards, we present two induction theorems, Theorem \ref{t22} and Theorem \ref{teo},
which are the analogues of \cite[Theorems 2.2 and 1.2]{AV09} in our case.

Finally, we use John's theorem and Proposition \ref{pconvex} to include a transformation of any convex body $\Delta$ into a homothecy of the standard $n$-simplex of comparable $n$-volume. By doing this, we are able to translate Theorem \ref{teo} and obtain Theorem \ref{TEO}. As a consequence of this, we obtain Corollary \ref{cor}, which proves both conjectures.

\subsection*{Acknowledgments}

I thank Francesco Amoroso and Martín Sombra for their advice, corrections and patience.
I also thank Éric Ricard for calling my attention to John's theorem.

\section{Preliminaries}

\subsection{Homomorphisms and subgroups of algebraic tori}
Let $\Gm^n=(\mathbb{C}^{\times})^n$ be the \emph{multiplicative group} or \emph{algebraic torus} of dimension $n$.
A point $(x_1,\ldots,x_n)\in\Gm^n$ is alternatively denoted by $\bfx$.
In particular,
$\mathbf{1}=(1,\ldots,1)$ represents the identity element.
Given $\bfx\in\Gm^n$ and $\boldsymbol{\lambda}=(\lambda_1,\ldots,\lambda_n)\in\mathbb{Z}^n$ we denote
\[
\bfx^{\boldsymbol{\lambda}}=x_1^{\lambda_1}\cdots x_n^{\lambda_n}
\mbox{.}
\]
Moreover, given $S\subset\Gm^n$ any subset we denote
\[
\bfx\cdot S=\lbrace \bfx\cdot\bfy \mid\; \bfy\in S\rbrace
\mbox{.}
\] 
If the context is clear, we write simply $\bfx S$.

We call \emph{homomorphism} an algebraic group homomorphism
$\varphi:\Gm^{n_1}\rightarrow\Gm^{n_2}$.
There is a bijection between integer matrices $\mathcal{M}_{n_2\times n_1}(\mathbb{Z})$ and homomorphisms $\Hom(\Gm^{n_1},\Gm^{n_2})$ defined as follows. Let $M\in\mathcal{M}_{n_2\times n_1}(\mathbb{Z})$ and let $\bflambda_1,\ldots,\bflambda_{n_2}\in\mathbb{Z}^{n_1}$ be the row vectors of~$M$, then
\begin{align*}
\varphi_M:\Gm^{n_1} & \longrightarrow \Gm^{n_2}\\
\bfx &\longmapsto (\bfx^{\bflambda_1},\ldots,\bfx^{\bflambda_{n_2}})
\end{align*} 
defines the corresponding homomorphism.
In particular, for any $l\in\mathbb{Z}$, we define the \emph{multiplication} by $l$ as the following endomomorphism:
\begin{align*}
[l]:\Gm^n &\longrightarrow \Gm^n\\
(x_1,\ldots,x_n) &\longmapsto (x_1^l,\ldots ,x_n^l)
\end{align*}
wich corresponds to the diagonal matrix $l\cdot\Id\in\mathcal{M}_{n\times n}(\mathbb{Z})$.

\vspace*{2mm}

We denote by $\zeta_k$ a primitive $k$-th root of unity, for any $k\in\mathbb{N}_{>0}$, and by
\[
\mu_{k}=\lbrace \zeta\in \Gm \mid\; \zeta^k=1\rbrace
\]
the subgroup of $k$-th roots of unity.
In particular, we denote by
\[
\mu_{\infty}=\bigcup_{k\in\mathbb{N}_{>0}}\mu_k
\]
the subgroup of roots of unity in $\Gm$.
Therefore,
\[
\mu_{\infty}^n
=\lbrace \boldsymbol{\xi}\in\Gm^n \mid\;
[k]\boldsymbol{\xi}=\mathbf{1}\mbox{ for some } k\in\mathbb{N}_{>0}\rbrace
\]
is the subgroup of the \emph{torsion points} of $\Gm^n$
and $\mu_k^n=\lbrace \bfxi\in\Gm^n\mid\; [k]\bfxi=\mathbf{1}\rbrace$ is the subgroup of $k$-torsion points of $\Gm^n$.
For any subvariety $V\subset\Gm^n$, we denote by $\Vtorss=V\cap\mu_{\infty}^n$ the set of torsion points on $V$
and by $\Vtors$ its Zariski closure in $\Gm^n$.
We call $\Vtors$ the \emph{torsion subvariety} of $V$.

By \emph{torsion coset} we understand a subvariety $\bfomega H\subset\Gm^n$,
where $H$ is an irreducible algebraic subgroup of $\Gm^n$ and
$\bfomega$ a torsion point.
Let $V$ be a subvariety of $\Gm^n$,
then we say that a torsion coset $\bfomega H$ is maximal in $V$ if
it is maximal by inclusion.
By the toric version of the Manin-Mumford conjecture (Laurent Theorem), $\Vtors$ is the union of torsion cosets in $V$,
hence we can write
\[
\Vtors=\bigcup_{
\substack{
\bfomega H\subset V\\
\mbox{\scriptsize torsion coset}
}
}
\bfomega H
\mbox{.}
\]
In fact, it is enough to take the maximal torsion cosets in $V$ in the index of the union.

Let $\Lambda$ be a subgroup of $\mathbb{Z}^n$.
We denote by $\Lambda^{\textrm{sat}}=(\Lambda\otimes_\mathbb{Z}\mathbb{R})\cap\mathbb{Z}^n$ the \emph{saturation} of $\Lambda$
and we call $[\Lambda^{\textrm{sat}}:\Lambda]$ the \emph{index} of $\Lambda$.
In particular, we say that $\Lambda$ is \emph{saturated} if $[\Lambda^{\textrm{sat}}:\Lambda]=1$.
For any subgroup $\Lambda$, we define the algebraic subgroup of $\Gm^n$ associated to $\Lambda$ as follows
\[
H_\Lambda=
\lbrace \bfx\in\Gm^n \mid\; \bfx^{\boldsymbol{\lambda}}=1 \mbox{, } \boldsymbol{\lambda}\in \Lambda\rbrace\mbox{.}
\]

The following result allows us to understand the relation between subgroups of $\mathbb{Z}^n$ and algebraic subgroups of $\Gm^n$.

\begin{theorem}\label{theoremAlgebraicSubgroups}
The map $\Lambda\mapsto H_\Lambda$ is a bijection between subgroups of $\mathbb{Z}^n$ and algebraic subgroups of $\Gm^n$.
A subgroup $H_\Lambda$ is irreducible if and only if $\Lambda$ is saturated.
Moreover, for any two subgroups $\Lambda$ and $\Lambda'$ we have $H_{\Lambda} \cdot H_{\Lambda'}=H_{\Lambda\cap \Lambda'}$.
\end{theorem}

\begin{proof}
See \cite[Proposition 3.2.7 and Theorem 3.2.19]{BG06}.
\end{proof}

\begin{corollary}\label{homomorphism}
Let $H$ be a subgroup of $\Gm^n$ of dimension $n-r$,
then there exists a surjective homomorphism
\[
\varphi : \Gm^n \longtwoheadrightarrow \Gm^{r}
\]
such that $\Ker(\varphi)=H$.
\end{corollary}

\begin{proof}
By Theorem \ref{theoremAlgebraicSubgroups}, there exists a unique lattice $\Lambda\subset\mathbb{Z}^n$ such that
\[
H=H_\Lambda=
\lbrace \bfx\in\Gm^n \mid\; \bfx^{\boldsymbol{\lambda}}=1 \mbox{, } \boldsymbol{\lambda}\in \Lambda\rbrace\mbox{.}
\]
Take the saturated subgroup
$\Lambda^{\perp}=\lbrace \bfx\in\mathbb{Z}^n \mid\; \langle \bfx , \bfy\rangle =0 \mbox{ for all } \bfy\in\Lambda\rbrace$,
so $H_{\Lambda^{\perp}}$ is irreducible, that is $H_{\Lambda^{\perp}}\simeq\Gm^r$.
Also by Theorem \ref{theoremAlgebraicSubgroups}, we have that $\Gm^n=H_{\lbrace 0\rbrace}=H\cdot H_{\Lambda^{\perp}}$
and $\varphi$ can be obtained as the following composition of homomorphisms:
\[
\varphi:
\Gm^n
=
H \cdot H_{\Lambda^{\perp}}
\longtwoheadrightarrow
H_{\Lambda^{\perp}}
\xlongrightarrow{\simeq}
\Gm^r
\mbox{.}
\]
\end{proof}

Let $V$ be a variety in $\Gm^n$, we define the \emph{stabilizer} of $V$ as
\[
\Stab(V)=\lbrace \boldsymbol{\xi}\in\Gm^n \mid\; \boldsymbol{\xi}V=V \rbrace
\mbox{.}
\]
In particular, $\Stab(V)$ is an algebraic subgroup of $\Gm^n$.
By means of Corollary \ref{homomorphism} we are able to identify $V$, via a homomorphism, to a variety with trivial stabilizer.
The following result is a direct consequence of Corollary \ref{homomorphism} and illustrates some useful properties of this homomorphism.

\begin{corollary}\label{remark}
Let $V\subset \Gm^n$ be a variety. Then there exists a homomorphism
$\varphi:\Gm^n\rightarrow\Gm^r$ such that $r=\codim(\Stab(V))$ and $\Ker(\varphi)=\Stab(V)$.
Moreover, $\varphi$ satisfies
\begin{itemize}
\item[(i)]
$\Stab(\varphi(V))=\lbrace \mathbf{1}\rbrace$;
\item[(ii)]
$\varphi^{-1}(\varphi(V))=V$; 
\item[(iii)]
$\varphi^{-1}(\bfeta) V= \bfeta_0 V$, for every $\bfeta\in\Gm^r$ and for any $\bfeta_0\in\varphi^{-1}(\bfeta)$.
\end{itemize}
\end{corollary}

An extra remark should be made regarding the relation between the torsion cosets and the stabilizer of $V$.
For any torsion coset $\bfomega H$ in $V$,
we have that $\Stab(V)\cdot\bfomega H$ is a union of torsion cosets in $V$.
In particular, every maximal torsion coset in $V$ has dimension at least $\dim(\Stab(V))$.

\subsection{Hilbert function}
Let $V$ be a variety in $\Gm^n$. We define the \emph{degree of definition} of $V$, $\delta(V)$, as the minimal degree $\delta$ such that $V$ is the intersection of hypersurfaces of degree at most $\delta$. We also define the \emph{degree of incomplete definition} of $V$, $\delta_0(V)$, as the minimal degree $\delta_0$ such that there exists an intersection $X$ of hypersurfaces of degree at most $\delta_0$ such that any irreducible component of $V$ is a component of $X$.
As a direct consequence of the definition, for any equidimensional variety $V$, we have the following inequalities
\[
\delta_0(V)\leq \delta(V)\leq \deg(V)\mbox{.}
\]

Let $V$ be a subvariety of $\Gm^n$ and
let the closure of $V$ in $\mathbb{P}^n$ be defined by the homogeneous radical ideal $I$ in
$\overline{\mathbb{Q}}\left[\bfx \right]$.
For $\nu\in\mathbb{N}$,
we denote by $H(V;\nu)$ the Hilbert function~$\dim(\overline{\mathbb{Q}}[\bfx]/I)_\nu$.

The following upper bound for the Hilbert function,
is a theorem of Chardin \cite{Cha88}.
\begin{theorem}\label{HilbertUpperBound}
Let $V\subseteq\Gm^n$ be an equidimensional variety of dimension $d=n-k$
and let $\nu\in\mathbb{N}$. Then
\[
H(V;\nu)\leq 
\begin{pmatrix}
\nu + d\\ d
\end{pmatrix}
\deg(V)
\mbox{.}
\]
\end{theorem}

On the other hand, as a consequence of a result of Chardin and Phillipon \cite[Corollaire 3]{CP99}
on Castelnuovo's regularity, we have the following lower bound for the Hilbert function:
\begin{theorem}\label{HilbertLowerBound}
Let $V\subseteq\Gm^n$ be an equidimensional variety of dimension $d=n-k$
and $m=k(\delta_0(V)-1)$.
Then, for any integer $\nu>m$, we have
\[
H(V;\nu)\geq
\begin{pmatrix}
\nu+d-m\\ d
\end{pmatrix}
\deg(V)
\mbox{.}
\]
\end{theorem}

In order to use these results in this article, we need effective upper bounds for $\delta_0(V')$ when $V'$ is a specific type of equidimensional variety.
Let us recall first an easy lemma for $\delta$.

\begin{lemma}\label{lemma36}
Let $X_1,\ldots,X_t$ be subvarieties of $\Gm^n$. Then
\[
\delta\big(\bigcup_{i=1}^t X_i\big)\leq \sum_{i=1}^t\delta(X_i)
\mbox{.}
\]
\end{lemma}
\begin{proof}
It is enough to prove it for $t=2$.
Let $X_1$ be defined by the polynomials $f_1,\ldots,f_r$ with $\deg(f_i)\leq\delta(X_1)$ and
equivalently let $X_2$ be defined $g_1,\ldots,g_s$ with $\deg(g_i)\leq\delta(X_2)$.
Then $X_1\cup X_2$ is defined by the polynomials $f_ig_j$ for $1\leq i\leq r$ and $1\leq j\leq s$.
\end{proof}

In general, this result is not true if we use $\delta_0$ instead of $\delta$.
To have a similar lemma for $\delta_0$, we must therefore consider more specific varieties.
The following is a variation of \cite[Lemma 2.5.]{AV12}.

\begin{lemma}\label{lemmadeltas}
Let $V$ be an irreducible subvariety of $\Gm^n$ defined over $\mathbb{K}$.
Let $T\subset\mu_\infty^n\times\Gal(\mathbb{K}/\mathbb{Q})$ be a finite family with $t$ elements.
Then
\[
\delta_0\big(
\bigcup_{(\bfg,\phi)\in T}
\bfg V^{\phi}
\big)
\leq
t\delta_0(V)
\mbox{.}
\]
\end{lemma}

\begin{proof}
We say that an irreducible variety $W\subset\Gm^n$ is imbedded in a variety $X\subset\Gm^n$ if $W\subset X$ but $W$ is not an irreducible component of $X$.

By definition of $\delta_0(V)$, there exists a variety $X$ such that $V$ is an irreducible component of $X$ and $\delta_0(V)=\delta(X)$.

Let $\langle T\rangle \subset\mu_\infty^n\times\Gal(\mathbb{K}/\mathbb{Q})$ be the group generated by $T$
and let $S=\lbrace (\bfg,\phi)\in \langle T\rangle \mid\; \bfg V^\phi$
is imbedded in
$X\rbrace$. Consider
\[
\widetilde{X}=X\cap\big(\bigcap_{(\bfg,\phi)\in S}\bfg^{-1}X^{\phi^{-1}}\big)
\mbox{.}
\]
We have that $V$ is an irreducible component of $\widetilde{X}$ and $\delta(\widetilde{X})=\delta(X)=\delta_0(V)$.
Moreover, no $\bfg V^\phi$ is imbedded in $\widetilde{X}$, for $(\bfg,\phi)\in \langle T\rangle$.
Assume by contradiction that there is a $\bfg V^\phi$ imbedded in $\widetilde{X}$.
Since $\widetilde{X}\subset X$, $\bfg V^\phi$ is imbedded in $X$ and so $(\bfg,\phi)\in S$.
By induction, we suppose $(\bfg^n,\phi^n)\in S$ for some $n\geq 1$.
Then $\widetilde{X}\subset \bfg^{-n}X^{\phi^{-n}}$ and so $\bfg V^\phi$ is imbedded in $\bfg^{-n}V^{\phi^{-n}}$;
which implies $(\bfg^{n+1},\phi^{n+1})\in S$.
Therefore, $(\bfg^n,\phi^n)\in S$ for every $n\in\mathbb{N}_{>0}$.
In particular, taking $n=\lcm(\ord(\bfg),\ord(\phi))$ we will have $(\mathbf{1},\Id)\in S$ which is a contradiction.

Next we define
\[
Y=\bigcup_{(\bfg,\phi)\in T}
\bfg \widetilde{X}^{\phi}
\mbox{.}
\]
Then $\bigcup_{(\bfg,\phi)}\bfg V^{\phi}\subset Y$ and $\delta(Y)\leq t\delta(\widetilde{X})=t\delta_0(V)$ by Lemma \ref{lemma36}.
Moreover, no $\bfg V^{\phi}$ is imbedded in $Y$, for $(\bfg,\phi)\in T$.
Assume by contradiction that there is a $(\bfg,\phi)\in T$ such that $\bfg V^{\phi}$ is imbedded in $Y$.
Then, there exists some $(\bfg_0,\phi_0)\in T$ such that $\bfg V^{\phi}$ is imbedded in $\bfg_0\widetilde{X}^{\phi_0}$.
Thus $\bfg_0^{-1}\bfg V^{\phi_0^{-1}\phi}$ is imbedded in $\widetilde{X}$ and, since $(\bfg_0^{-1}\bfg,\phi_0^{-1}\phi)\in G$, this contradicts the definition of $\widetilde{X}$.
\end{proof}

Let $V\subset\Gm^n$ be any subvariety.
We say that $V$ is \emph{minimally defined} over $\mathbb{K}$,
if $\mathbb{K}$ is the minimal Galois extension of $\mathbb{Q}$ such that $V$ is defined over $\mathbb{K}$.

If $\mathbb{K}$ is an abelian extension, by the Kronecker-Weber theorem, we have
that $\mathbb{K}$ is contained in a cyclotomic extension of $\mathbb{Q}$.
In fact, there is a unique minimal cyclotomic extension $\mathbb{Q}(\zeta_N)$ containing $\mathbb{K}$
\cite[Theorem 4.27(v)]{Nar04}.
If $N\equiv 2\pmod 4$, then $\mathbb{Q}(\zeta_N)=\mathbb{Q}(\zeta_{N/2})$.
Therefore, we can always choose $N\not\equiv 2\pmod 4$.

The following lemma is a key ingredient in the proof of Theorem \ref{t21}.

\begin{lemma}\label{lemmaHilbert}
Let $V\subseteq\Gm^n$ be an irreducible variety of dimension $d=n-k$,
minimally defined over $\mathbb{K}$.
Let $\phi\in\Gal(\mathbb{K}/\mathbb{Q})$
and let $\bfe\in\mu_\infty^n$.
\begin{itemize}
\item[(a)]
If $\bfe V^\phi\neq  V$,
then there exists a homogeneous polynomial $F$ of degree at most $2k(2d+1)\delta_0(V)$
such that $F\equiv 0$ in $\bfe V^\phi$ and $F\not\equiv 0$ in $V$.
\item[(b)]
If $[2]^{-1}(\bfe V^\phi)\neq V$,
then there exists a homogeneous polynomial $G$ of degree at most $2^n k(2d+1)\delta_0(V)$
such that $G\equiv 0$ in $[2]^{-1}(\bfe V^\phi)$ and $G\not\equiv 0$ in $V$.
\end{itemize}
\end{lemma}

\begin{proof}
\mbox{}
\begin{itemize}
\item[(a)]
Since $V$ is an irreducible variety, $\bfe V^\phi$ is also irreducible and of the same degree.
By Theorem \ref{HilbertUpperBound} we get, for any $\nu\in\mathbb{N}$,
\[
H(\bfe V^\phi;\nu)\leq
	\begin{pmatrix}
		\nu+d\\ d
	\end{pmatrix}
\deg(V)
\mbox{.}
\]

On the other hand, let $V'=V\cup \bfe V^\phi$.
This is a $d$-equidimensional variety of degree $2\deg(V)$.
Thereby, using Theorem \ref{HilbertLowerBound} we have, for any $\nu>m$,
\[
H(V';\nu)\geq
	\begin{pmatrix}
		\nu+d-m\\d
	\end{pmatrix}
2\deg(V)
\mbox{,}
\]
where $m=k(\delta_0(V')-1)$.
In particular, $m\leq 2k\delta_0(V)$ due to Lemma \ref{lemmadeltas}.

Taking $\nu=m(2d+1)$ we obtain
\[
\begin{pmatrix}
	\nu+d\\ d
\end{pmatrix}
\begin{pmatrix}
	\nu+d-m\\ d
\end{pmatrix}^{-1}
\leq
\left(
	1+\frac{m}{\nu-m}
\right)^d
=
\left(
	1+\frac{1}{2d}
\right)^d
\leq
\mathrm{e}^{1/2}
<2
\mbox{.}
\]
Hence, $H(\bfe V^\phi;\nu)<H(V';\nu)$.

This means that there exists a homogeneous polynomial $F$ of degree $\nu$ such that
$F\equiv 0$ on $\bfe V^\phi$ and $F\not\equiv 0$ on $V'=\bfe V^\phi \cup V$.
In particular $F\not\equiv 0$ on $V$.
Moreover, $\deg(F)=\nu\leq 2k(2d+1)\delta_0(V)$,
which proves (a).

\vspace*{2mm}

\item[(b)]
Let $W=[2]^{-1}(\bfe V^\phi)$.
This is a $d$-equidimensional variety of degree $2^k\deg(V)$.
By Theorem \ref{HilbertUpperBound} we get, for any $\nu\in\mathbb{N}$,
\[
H(W;\nu)\leq
	\begin{pmatrix}
		\nu+d\\ d
	\end{pmatrix}
2^k\deg(V)
\mbox{.}
\]

On the other hand,
let $E=\lbrace \bfe_0\in\mu_\infty^n \mid\; \bfe_0^2\in\Stab(V)\rbrace$
and let $W'=E\cdot V$.
This is also a $d$-equidimensional variety of degree $2^r\deg(V)$, for some $k<r\leq n$.
That is because $E/\Stab(V)\simeq\mu_2^r$
(see Corollary \ref{remark}, with $r=\codim(\Stab(V))$ and $E=\varphi^{-1}(\mu_2^r)$).
Thereby, using Proposition \ref{HilbertLowerBound} we have, for any $\nu>m$,
\[
H(W';\nu)\geq
	\begin{pmatrix}
		\nu+d-m\\d
	\end{pmatrix}
2^r\deg(V)
\mbox{,}
\]
where $m=k(\delta_0(W')-1)$.
In particular, $m\leq 2^{n}k\delta_0(V)$ due to Lemma \ref{lemmadeltas}.

Taking $\nu=m(2d+1)$, we obtain:
\[
\begin{pmatrix}
	\nu+d\\ d
\end{pmatrix}
\begin{pmatrix}
	\nu+d-m\\ d
\end{pmatrix}^{-1}
\leq
\mathrm{e}^{1/2}
<2^{r-k}
\mbox{.}
\]
Hence, $H(W;\nu)<H(W',\nu)$.

This means that there exists a homogeneous polynomial $\widetilde{G}$ of degree $\nu$ such that
$\widetilde{G}\equiv 0$ on $W=[2]^{-1}(\bfe V^\phi)$ and
$\widetilde{G}\not\equiv 0$ on $W'=E\cdot V$.
In particular, there exists $\bfe_0\in E$ such that
$\widetilde{G}\not\equiv 0$ on $\bfe_0 V$.

Let $G(x)=\widetilde{G}(\bfe_0^{-1} x)$.
We have that $G\equiv 0$ on $\bfe_0[2]^{-1}(\bfe V)$.
By definition of $E$, $\bfe_0\in\Stab([2]^{-1}(\bfe V))$.
Hence, $G\equiv 0$ on $[2]^{-1}(\bfe V)$.
We also have that
$G\not\equiv 0$ on $[2]\bfe_0\cdot V$.
Since $[2]\bfe_0\in\Stab(V)$, this implies  $G\not\equiv 0$ on $V$.
Moreover, $\deg(G)=\nu\leq 2^nk(2d+1)$,
which proves (b).
\end{itemize}
\end{proof}

\section{Algebraic interpolation}
We generalize \cite[Lemma 1]{BS02} to general dimension $n$ and to any abelian extension of $\mathbb{Q}$ with the following result.
\begin{lemma}\label{lemma21}
Let $V\subset\Gm^n$ be a variety defined over $\mathbb{Q}(\zeta_N)$, with $N\not\equiv 2\pmod{4}$,
and let $\bfomega\in V$ be a torsion point.
\begin{itemize}

\item[1.]
If $4\nmid N$,
then one of the following is true:
\begin{itemize}
\item[(a)]
there exists $\bfeta\in\mu_2^n\setminus\lbrace \mathbf{1}\rbrace$ such that $\bfeta\cdot \bfomega\in V$;
\item[(b)]
there exists $\bfeta\in\mu_2^n$ such that $\bfeta\cdot [2]\bfomega\in V^{\sigma}$,
where $\sigma\in\Gal(\mathbb{Q}(\zeta_N)/\mathbb{Q})$ maps $\zeta_N\mapsto\zeta_N^2$.
\end{itemize}

\item[2.]
If $N=4N'$,
then one of the following is true:
\begin{itemize}
\item[(c)]
there exists $\bfeta\in\mu_2^n\setminus\lbrace \mathbf{1}\rbrace$ such that $\bfeta\cdot \bfomega\in V$;
\item[(d)]
there exists $\bfeta\in\mu_2^n$ such that $\bfeta\cdot \bfomega\in V^{\tau}$,
where $\tau\in\Gal(\mathbb{Q}(\zeta_N)/\mathbb{Q})$ maps $\zeta_N\mapsto\zeta_N^{1+2N'}$.
\end{itemize}

\end{itemize}
\end{lemma}

Note that the case $\mathbb{K}=\mathbb{Q}$ is included in case 1,
corresponding to $N=1$
(so $\sigma$ is the identity).

\begin{proof}
Let $l$ be the order of $\bfomega$, in particular $\bfomega\in\mathbb{Q}(\zeta_l)$.

\vspace*{2mm}

1.\quad
By hypothesis, $N$ is odd.
Let $M=\lcm (N,l)$.
We distinguish $3$ cases.

\begin{itemize}
\item[(i)]
If $l=4l'$, then $M=4M'$.
In particular, we have $\gcd(1+2M',M)=1$.
Therefore, we can take a Galois automorphism $\tau\in\Gal(\mathbb{Q}(\zeta_M)/\mathbb{Q})$
mapping $\zeta_M\mapsto\zeta_M^{1+2M'}$.
Since  $2M'\equiv 2l'\pmod l$,
we have that $\tau$ maps $\zeta_l\mapsto\zeta_l^{1+2l'}$.
On the other hand, $N$ is odd so $N|M'$
and $\zeta_N$ is invariant by the action of $\tau$.
Hence $V^\tau=V$ and $[1+2l']\bfomega\in V$.
Choosing $\bfeta=[2l']\bfomega\in\mu_2^n\setminus\lbrace\mathbf{1}\rbrace$, (a) holds.

\item[(ii)]
If $l=2l'$ with $2\nmid l'$,
then $M=2M'$ with $2\nmid M'$.
In particular, we have $\gcd(2+M',M)=1$.
Therefore, we can extend $\sigma$ to a Galois automorphism in $\Gal(\mathbb{Q}(\zeta_M)/\mathbb{Q})$,
mapping $\zeta_M\mapsto\zeta_M^{2+M'}$
(this extends $\sigma$ because $N$ is odd and so $N\mid M'$).
Since $M'\equiv l'\pmod l$,
we have that $\sigma$ maps $\zeta_l\mapsto\zeta_l^{2+l'}$.
Hence $[2+l']\bfomega\in V^\sigma$.
Choosing $\bfeta=[l']\bfomega\in\mu_2^n\setminus\lbrace \mathbf{1}\rbrace$, (b) holds.

\item[(iii)]
If $2\nmid l$, then $2\nmid M$.
We have that $\sigma$ can be extended to a Galois automorphism in $\Gal(\mathbb{Q}(\zeta_M)/\mathbb{Q})$
mapping $\zeta_M\mapsto \zeta_M^2$.
In particular, $\sigma$ maps $\zeta_l\mapsto\zeta_l^2$.
Hence $[2]\bfomega\in V^\sigma$.
Choosing $\bfeta=\mathbf{1}$, (b) holds.
\end{itemize}

\vspace*{2mm}

2.\quad
Let $M=4M'=\lcm(N,l)$, and
$\widetilde{\tau}$ be an automorphism in $\Gal(\mathbb{Q}(\zeta_M)/\mathbb{Q})$
mapping $\zeta\mapsto\zeta_M^{2M'+1}$. 
We distinguish $2$ cases.
\begin{itemize}

\item[(i)]
If $N\mid 2M'$, then $l\nmid 2M'$ (otherwise, we would have $\lcm(N,l)=2M'$).
Since $2M'\equiv 2l'\pmod l$,
we have that $\widetilde{\tau}$ maps $\zeta_l\mapsto\zeta_l^{1+2l'}$.
On the other hand, $2M'\equiv 0\pmod N$
and so $\widetilde{\tau}$ maps $\zeta_N\mapsto\zeta_N$.
Hence $V^{\widetilde{\tau}}=V$ and $[1+2l']\bfomega\in V$.
Choosing $\bfeta=[2l']\bfomega\in\mu_2^n\setminus\lbrace\mathbf{1}\rbrace$, (c) holds.

\item[(ii)]
If $N\nmid 2M'$, then $2N'\equiv 2M'\pmod N$.
We have that $\widetilde{\tau}$ maps $\zeta_N\mapsto\zeta_N^{1+2N'}$ thus
$\widetilde{\tau}_{\mid \mathbb{Q}(\zeta_N)}=\tau$.
Hence $[2M'+1]\bfomega=[2l'+1]\bfomega\in V^{\tau}$.
Choosing $\bfeta=[2M']\bfomega\in\mu_2^n$, (d) holds.
\end{itemize}

\end{proof}

As a consequence of this lemma, for any irreducible variety $V$ we can find another variety $V'$ containing the torsion subvariety of $V$ but not containing $V$.

\begin{proposition}\label{proposition21}
Let $V\subset\Gm^n$ be an irreducible variety,
minimally defined over $\mathbb{K}$ such that $\Vtors\neq V$.
Let $r=\codim (\Stab(V))$ and let $\varphi :\Gm^n\rightarrow\Gm^r$ be a homomorphism\footnote{This homomorphism exists by Corollary \ref{homomorphism}.} such that $\Stab(V)=\Ker(\varphi )$.

\begin{itemize}
\item[1.]
If $\mathbb{K}$ is abelian and $\mathbb{Q}(\zeta_N)$ is a cyclotomic extension of $\mathbb{K}$,
with $4\nmid N$.
Then 
\[
\Vtors\subset V'=
\displaystyle\bigcup_{\bfeta\in\mu_2^r\setminus\lbrace \mathbf{1}\rbrace} (\varphi^{-1}(\bfeta) V)
\cup
\displaystyle\bigcup_{\bfeta\in\mu_2^r} [2]^{-1}(\varphi^{-1}(\bfeta) V^\sigma)
\mbox{,}
\]
where $\sigma\in\Gal(\mathbb{Q}(\zeta_N)/\mathbb{Q})$, mapping $\zeta_N\mapsto\zeta_N^2$.
Moreover $V'\cap V\subsetneq V$.

\item[2.]
If $\mathbb{K}$ is abelian and $\mathbb{Q}(\zeta_N)$ is its minimal cyclotomic extension,
with $N=4N'$.
Then 
\[
\Vtors\subset V'=
\displaystyle\bigcup_{\bfeta\in\mu_2^r\setminus\lbrace \mathbf{1}\rbrace} (\varphi^{-1}(\bfeta) V)
\cup
\displaystyle\bigcup_{\bfeta\in\mu_2^r} (\varphi^{-1}(\bfeta) V^\tau)
\mbox{,}
\]
where $\tau\in\Gal(\mathbb{Q}(\zeta_N)/\mathbb{Q})$, mapping $\zeta_N\mapsto\zeta_N^{1+2N'}$.
Moreover $V'\cap V\subsetneq V$.

\item[3.]
If $\mathbb{K}$ is not abelian.
Then $\Vtors\subset V\cap V^\varsigma \subsetneq V$,
for any $\varsigma\in\Gal(\mathbb{K}/\Qab\cap\mathbb{K})$
such that $\varsigma\neq\Id$.
\end{itemize}
\end{proposition}

Note that the $V'$ in the proposition are finite unions of varieties.
That is because, using Corollary \ref{remark}(iii), for each $\bfeta\in\mu_2^r$ it suffices to take just one preimage $\bfeta_0\in\varphi^{-1}(\bfeta)$ instead of the the whole $\varphi^{-1}(\bfeta)$.

\begin{proof}
\mbox{}
1.\quad
Let $\Vtorss=\lbrace \bfomega'\in V \mid \bfomega' \mbox{ torsion point}\rbrace$.
It is enough to see that $\Vtorss\subset V'$ to prove $\Vtors\subset V'$.
To show this, we take $\bfomega'\in \Vtorss$ and we have that $\varphi(\bfomega')$ is a torsion point in $\varphi(V)$.
Since $\varphi(V)$
is defined over $\mathbb{Q}(\zeta_N)$ with $N$ odd,
we can apply point 1 in Lemma \ref{lemma21} to $\varphi(V)$, hence one of the following is true:

\begin{itemize}
\item[(a)]
There exists $\bfeta\in\mu_2^r\setminus\lbrace\mathbf{1}\rbrace$ such that
$\varphi(\bfomega')\in \bfeta \varphi(V)$.
Hence $\bfomega'\in\varphi^{-1}(\bfeta) V$.
By definition of $\varphi$ we have that
$\varphi^{-1}(\bfeta)\not\in\Stab(V)$ and so
$(\varphi^{-1}(\bfeta) V)\cap V\subsetneq V$.

\item[(b)]
There exists $\bfeta\in\mu_2^r$ such that
$[2](\varphi(\bfomega'))\in \bfeta \varphi(V)^\sigma$.
So $\bfomega'\in[2]^{-1}(\varphi^{-1}(\bfeta) V^\sigma)$.
Moreover, we have that $[2]^{-1}(\varphi^{-1}(\bfeta) V^\sigma)\cap V\subsetneq V$.
We prove this by contradiction.
Assume $V\subset [2]^{-1}(\varphi^{-1}(\bfeta) V^\sigma)$:
this means that for every $\bfeta_0\in\mu_2^r$ we have $\varphi^{-1}(\bfeta_0)\in\Stab([2]^{-1}(\varphi^{-1}(\bfeta) V^\sigma)$ and therefore
\[
\bigcup_{\bfeta_0\in\mu_2^r}\varphi^{-1}(\bfeta_0) V\subset [2]^{-1}(\varphi^{-1}(\bfeta) V^\sigma)
\mbox{.}
\]
Since $\varphi^{-1}(\mu_2^r)\cap\Stab(V)=\lbrace\mathbf{1}\rbrace$ (see (a) above) and all the translates of $V$ have the same stabilizer,
we have that that $\deg(\bigcup\varphi^{-1}(\bfeta_0) V)=2^r\deg(V)$.
On the other hand, we have $\deg([2]^{-1}(\varphi^{-1}(\bfeta) V^\sigma))=2^k\deg(V)$, see \cite[Lemme~6(i)]{Hin88}.
Since $k<r$, this leads to a contradiction.
\end{itemize}

Thereby $\bfomega'\in V'$ for every $\bfomega'\in \Vtorss$ and thus $\Vtors\subset V'$.
Moreover, $V'\cap V\subsetneq V$.

2.\quad
The proof follows similar to the previous one, using point 2 in Lemma \ref{lemma21}.
Note that in this case
we need the minimality of $N$ to guarantee that $V^\tau\neq V$ and therefore that $V\cap V'\neq V$.

3.\quad
Every torsion coset is defined over $\Qab$ and $\Qab$ is invariant by $\varsigma$.
Hence, for every torsion coset $\bfomega H\subset V$,
we have that $\bfomega H\subset V^{\varsigma}$ and $V\neq V^\varsigma$,
due to the minimality of $\mathbb{K}$.
\end{proof}

\begin{Remark}\label{rmk}
We can consider $\mathbb{K}=\mathbb{C}$ as the field of definition of $V$.
Then it follows equivalently to the case when $\mathbb{K}$ is not abelian to prove that
\[
\Vtors\subset V\cap V^\varsigma\subsetneq V
\mbox{,}
\]
for any $\varsigma\in\Gal(\mathbb{C}/\Qab)$ such that $\varsigma\neq \Id$.
\end{Remark}

The following theorem is a specialization of \cite[Theorem 1.2]{AV09} to torsion subvarieties.
Keeping the notation of Proposition \ref{proposition21},
note that $\varphi^{-1}(\bfeta)^2$ lies in the stabilizer of $V$ for any $\bfeta\in\mu_\infty^r$.
This is a fundamental observation so that we can use Lemma \ref{lemmaHilbert} in the proof of the theorem.

\begin{theorem}\label{t21}
Let $V\subset \Gm^n$ be an irreducible variety of dimension $d$ and codimension $k$.
We assume that $V$ is not a torsion coset.
Let
\[
\theta_0=\theta_0(V)=k(2^{2n}+2^{n+1}-2)(2d+1)\delta_0(V)
\mbox{.}
\]
Then $\Vtors$ is contained in a hypersurface $Z$ of degree at most $\theta_0$, which does not contain~$V$;
that is $\Vtors\subset V\cap Z\varsubsetneq V$.
\end{theorem}

\begin{proof}
Let $V$ be minimally defined over $\mathbb{K}$.
To prove this theorem, we distinguish three cases, according to Proposition \ref{proposition21}.

\begin{itemize}
\item[1.]
If $\mathbb{K}$ is abelian and $\mathbb{Q}(\zeta_N)$ is a cyclotomic extension of $\mathbb{K}$,
with $4\nmid N$,
then by point 1 in Proposition \ref{proposition21} we have that
\[
\Vtors\subset V'=
\displaystyle\bigcup_{\bfeta\in\mu_2^r\setminus\lbrace \mathbf{1}\rbrace} (\varphi^{-1}(\bfeta) V)
\cup
\displaystyle\bigcup_{\bfeta\in\mu_2^r} [2]^{-1}(\varphi^{-1}(\bfeta) V^\sigma)
\mbox{,}
\]
where $\sigma\in\Gal(\mathbb{Q}(\zeta_N)/\mathbb{Q})$, mapping $\zeta_N\mapsto\zeta_N^2$
and $V\cap V'\subsetneq V$.

To prove the theorem we find, for each component of $V'$, a hypersurface containing it, but not containing $V$.
To conclude, it is enough to take $Z$ as the union of these hypersurfaces
and the only thing left to check is the degree of $Z$.

For every $\bfeta\in\mu_2^r\setminus\lbrace\mathbf{1}\rbrace$,
we choose an $\bfe\in\varphi^{-1}(\bfeta)$ and $\phi=\Id$.
We use Lemma \ref{lemmaHilbert}(a)
and we obtain a homogeneous polynomial
$F_{\bfeta}$ of degree at most $2k(2d+1)\delta_0(V)$ such that
$F_{\bfeta}\equiv 0$ on $\varphi^{-1}(\bfeta) V$ and
$F_{\bfeta}\not\equiv 0$ on $V$.

On the other hand,
for every $\bfeta\in\mu_2^r$,
we choose a $\bfe\in\varphi^{-1}(\bfeta)$ and $\phi=\sigma$.
We use  Lemma \ref{lemmaHilbert}(b)
and we obtain a homogeneous polynomial
$G_{\bfeta}$ of degree at most $2^nk(d+1)\delta_0(V)$ such that
$G_{\bfeta}\equiv 0$ on $[2]^{-1}(\varphi^{-1}(\bfeta) V)$ and
$G_{\bfeta}\not\equiv 0$ on $V$.

For $Z\subset\Gm^n$ the hypersuface defined by
\[
\displaystyle\prod_{\bfeta\in\mu_2^r\setminus\lbrace\mathbf{1}\rbrace} F_{\bfeta} (\bfx)
\cdot
\displaystyle\prod_{\bfeta\in\mu_2^r} G_{\bfeta}(\bfx)=0
\mbox{,}
\]
we have
\[
\deg(Z)\leq
\displaystyle\sum_{\bfeta\in\mu_2^r\setminus\lbrace\mathbf{1}\rbrace} 2k(2d+1)\delta_0(V)
+
\displaystyle\sum_{\bfeta\in\mu_2^r} 2^nk(2d+1)\delta_0(V)
\leq \theta_0
\]
and $\Vtors\subset V\cap V'\subset V\cap Z\subsetneq V$.

\vspace*{2mm}

\item[2.]
If $\mathbb{K}$ is abelian and $\mathbb{Q}(\zeta_N)$ is its minimal cyclotomic extension,
with $N=4N'$,
then by point 2 in Proposition \ref{proposition21} we have
\[
\Vtors\subset V'=
\displaystyle\bigcup_{\bfeta\in\mu_2^r\setminus\lbrace \mathbf{1}\rbrace} (\varphi^{-1}(\bfeta) V)
\cup
\displaystyle\bigcup_{\bfeta\in\mu_2^r} (\varphi^{-1}(\bfeta) V^\tau)
\mbox{,}
\]
where $\tau\in\Gal(\mathbb{Q}(\zeta_N)/\mathbb{Q})$, mapping $\zeta_N\mapsto\zeta_N^{1+2N'}$ and
$V\cap V'\subsetneq V$.
We proceed as before.

For every $\bfeta\in\mu_2^r\setminus\lbrace\mathbf{1}\rbrace$,
we choose an $\bfe\in\varphi^{-1}(\bfeta)$ and $\phi=\Id$.
We use Lemma~\ref{lemmaHilbert}(a)
and we obtain a homogeneous polynomial
$F_{\bfeta}$ of degree at most $2k(2d+1)\delta_0(V)$ such that
$F_{\bfeta}\equiv 0$ on $\varphi^{-1}(\bfeta) V$ and
$F_{\bfeta}\not\equiv 0$ on $V$.

On the other hand,
for every $\bfeta\in\mu_2^r$,
we choose a $\bfe\in\varphi^{-1}(\bfeta)$ and $\phi=\tau$.
We use again Lemma \ref{lemmaHilbert}(a)
and we obtain a homogeneous polynomial
$F'_{\bfeta}$ of degree at most $2k(2d+1)\delta_0(V)$ such that
$F'_{\bfeta}\equiv 0$ on $\varphi^{-1}(\bfeta) V^\tau$ and
$F'_{\bfeta}\not\equiv 0$ on $V$.

For $Z\subset\Gm^n$ the hypersurace defined by
\[
\displaystyle\prod_{\bfeta\in\mu_2^r\setminus\lbrace\mathbf{1}\rbrace} F_{\bfeta}(\bfx)
\cdot
\displaystyle\prod_{\bfeta\in\mu_2^r} F'_{\bfeta}(\bfx)
=0
\mbox{,}
\]
we have
\[
\deg(Z)\leq
\displaystyle\sum_{\bfeta\in\mu_2^r\setminus\lbrace\mathbf{1}\rbrace} 2k(2d+1)\delta_0(V)
+
\displaystyle\sum_{\bfeta\in\mu_2^r} 2k(2d+1)\delta_0(V)
\leq \theta_0
\]
and $\Vtors\subset V\cap V'\subset V\cap Z\subsetneq V$.

\vspace*{2mm}

\item[3.]
If $\mathbb{K}$ is not abelian,
by point 3 in Proposition \ref{proposition21},
we have that $\Vtors\subset V\cap V^{\varsigma}$
for any $\varsigma\in\Gal(\mathbb{K}/\Qab\cap\mathbb{K})$ such that $\varsigma\neq \Id$.

We choose $\bfe=\mathbf{1}$ and $\phi=\varsigma$.
we use Lemma \ref{lemmaHilbert}(a)
and we obtain a homogeneous polynomial $F$ of degree at most $2k(2d+1)\delta_0(V)$ such that
$F\equiv 0$ on $V^{\varsigma}$ and
$F\not\equiv 0$ on $V$.
Therefore, if we take the hypersurface $Z$ defined by $F(\bfx)=0$, we have $\deg(Z)\leq\theta_0$ and
$\Vtors\subset V\cap V^{\varsigma}\subset V\cap Z\subsetneq V$.
\end{itemize}

If $V$ is not defined over an extension of $\mathbb{Q}$, we use Remark \ref{rmk} and the proof follows as point $3$.
\end{proof}

\section{Induction theorems}

The following theorems correspond to Theorem 2.2 and Theorem 1.2 in \cite{AV09}.
Their proofs follow similarly to the ones of Amoroso and Viada.
For the convenience of the reader, we reproduce the proofs.

\begin{theorem}\label{t22}
Let $V_0\subset V_1$ be subvarieties of $\Gm^n$ of codimension $k_0$ and $k_1$ respectively and $V_0$ irreducible.
Let
\[
\theta=\theta(V_1)=((2n-1)k_0(2^{2n}+2^{n+1}-2))^{k_0-k_1+1}\delta(V_1)
\mbox{.}
\]
Then one of the following is true:
\begin{itemize}
\item[(a)]
there exists a torsion coset $B$ such that $V_0\subseteq B\subseteq V_1$;
\item[(b)]
there exists a hypersurface $Z$ of degree at most $\theta$ such that $V_0\nsubseteq Z$ and $\overline{V_{0,\textrm{tors}}}\subseteq Z$. 
\end{itemize}
\end{theorem}

\begin{proof}
We assume the statement to be false, that is:
\begin{itemize}
\item[(a')]
$V_0$ is not contained in any torsion coset $B\subset V_1$;
\end{itemize}
and
\begin{itemize}
\item[(b')]
for every hypersurface $Z$ satisfying $\overline{V_{0,\textrm{tors}}}\subset Z$, we also have that$V_0\subset Z$.
\end{itemize}

We define, for $r=0,\ldots,k_0-k_1+1$,
\[
D_r=((2n-1)k_0(2^{2n}+2^{n+1}-2))^{r}\delta(V_1)
\]
and we build a chain of varieties
\[
X_0=V_1\supseteq \cdots \supseteq X_{k_0-k_1+1}
\]
such that, for every $r=0,\ldots,k_0-k_1+1$, the following hold:
\begin{itemize}
\item[(i)]
$V_0\subseteq X_r$,
\item[(ii)]
each irreducible component of $X_r$ containing $V_0$ has codimension at least $r+k_1$,
\item[(iii)]
$\delta(X_r)\leq D_r$.
\end{itemize}

If this holds for $r=k_0-k_1+1$, we have a component of $X_r$ of codimension at least $k_0+1$ containing $V_0$
which is a contradiction.

We build the chain by recursion:
\begin{itemize}
\item
For $r=0$, $X_0=V_1$ satisfies the properties.
\item
For $r+1>1$, we assume we have already constructed $X_r$.
Let $W_1,\ldots,W_t$ be the irreducible components of $X_r$ such that
\[
V_0\subset W_j\;\Leftrightarrow\; 1\leq j\leq s\mbox{.}
\] 
Property (i) guarantees that $s>0$ and, together with property (ii), we have that $r+k_1\leq\codim(W_j)\leq k_0$ for $1\leq j\leq s$.

For every $j=1,\ldots,s$ 
we have $\delta_0(W_j)\leq \delta(X_r)\leq D_r\leq\theta$ and $V_0\subseteq W_j\subset V_1$,
with $\codim(W_j)=k$.
Hence, by hypothesis (a'), $W_j$ is not a torsion coset and we can apply Theorem \ref{t22} to $W_j$.
Let $Z_j$ be the hypersurface of degree at most
$(2d+1)k(2^{2k}+2^{k+1}-2)\delta_0(W_j)\leq (2n-1)k_0(2^{2k_0}+2^{k_0+1}-2)D_r=D_{r+1}$
such that $\overline{W_{j,\textrm{tors}}}\subseteq W_j\cap Z_j\varsubsetneq W_j$.
Since $V_0\subseteq W_j$,
$\overline{V_{0,\textrm{tors}}}\subseteq \overline{W_{j,\textrm{tors}}}\subset Z_j$
and $\deg (Z_j)\leq D_{r+1}\leq \theta$,
hypothesis (b') guarantees that $V_0\subset Z_j$.

So
\[
V_0\subset\displaystyle\bigcap_{j=1}^{s}Z_j
\mbox{,}
\]
and we define
\[
X_{r+1}=X\cap\displaystyle\bigcap_{j=1}^{s}Z_j
\mbox{.}
\]
In particular, $V_0\subseteq X_{r+1}$ and property (i) is satisfied.
Moreover, property (iii) is also satisfied, because
\[
\delta(X_{r+1})
=\max \lbrace \delta(X_r),\deg Z_1,\ldots,\deg Z_s\rbrace
\leq\max \lbrace D_r, D_{r+1}\rbrace
=D_{r+1}
\mbox{.}
\]
By taking $W'_j=W_j\cap Z_1\cap \cdots \cap Z_s$ for every $j=1,\ldots,t$ we have
\[
X_{r+1}=W'_1\cup \cdots\cup W'_s\cup W'_{s+1}\cup W'_t
\mbox{.}
\]

For $j=1,\ldots,s$, we have that every irreducible component of $W'_j$ has codimension at least
$\codim(W_j)+1\geq r+k_1+1$.
And for $j=s+1,\ldots,t$ we have that $V_0\not\subset W_j$, thus $V_0$ is not contained in any irreducible component of $W'_j$.
This shows that property (ii) is satisfied.
\end{itemize}
\end{proof}

\begin{theorem13}\label{t12}
Let $V\subset\Gm^n$ be a variety of dimension  $d$.
For $j=0,\ldots,d$\;, let $\Vtorss^j$ be the $j$-equidimensional part of $\Vtors$.
Then
\[
\deg(\Vtorss^j)\leq c_{n,j}\delta(V)^{n-j}
\]
for every $i=0,\ldots,d$\;, where
\[
c_{n,j}=((2n-1)(n-1)(2^{2n}+2^{n+1}-2))^{d(n-j)}
\mbox{.}
\]
\end{theorem13}

\begin{proof}
If $\dim(V)=0$, $\card(\Vtors)\leq\deg(V)\leq \delta(V)^n$ and we are done. Hence, we suppose $\dim(V)>0$.

Let
\[
\theta=\theta(V)=((2n-1)(n-1)(2^{2n}+2^{n+1}-2))^{d}\delta(V)
\mbox{.}
\]
Observe that $c_{n,j}\delta(V)^{n-j}=\theta^{n-j}$.

Let $V=X^d\cup\cdots\cup X^0$, where $X^j$ represents the $j$-equidimensional part of $V$ for every~$j$.
We have $\Vtors=\Vtorss^d\cup\cdots\cup\Vtorss^0$ and we build the family $\Vtorss^d,\ldots,\Vtorss^0$ recursively as follows:

\textsc{Claim.}
For every $r=d,\ldots,0$
there exist an $r$-equidimensional varietie $Y^r$, such that
\begin{itemize}
\item[(i)]
$\Vtors\subseteq \Vtorss^d\cup\cdots\cup\Vtorss^{r+1}\cup Y^r \cup X^{r-1}\cup \cdots\cup X^{0}$;
\item[(ii)]
$\sum_{i=d}^{r+1}\theta^{i-r}\deg(\Vtorss^i)+\deg(Y^r)
\leq
\sum_{i=d}^{r}\theta^{i-r}\deg(X^i)$.
\end{itemize}

If the claim holds for $r=0$, by assertion (i) we have $\Vtorss^0\subset Y^0$.
Moreover, assertion (i) also guarantees that $\Vtorss^r\subset Y^r$ which,
using assertion (ii), implies
\[
\displaystyle\sum_{i=d}^{r}\theta^{i-r}\deg(\Vtorss^i)
\leq
\displaystyle\sum_{i=d}^{r}\theta^{i-r}\deg(X^i)
\]
A result of Philippon~\cite[Corollaire~5]{Phi95} (with $m=n$ and $S=\mathbb{P}^n$) shows that,
for~$\theta \geq \delta(V)$, we have
\[
\displaystyle\sum_{i=r}^{d}\theta^{i-r}\deg(X^i)\leq \theta^{n-r}
\mbox{.}
\]
Hence, setting $r=0$, we obtain
\[
\displaystyle\sum_{i=0}^d\theta^{i}\deg (\Vtorss^i)\leq\theta^n
\]
and the inequalities of the statement follow trivially.

It remains to prove the claim. We build the family as follows.
\begin{itemize}
\item
For $r=d$, we take $Y^d=X^d$ and the claim holds.
\item
Let $d\geq r> 0$. We assume that we have $Y^r$ satisfying the claim.
If $Y^r$ has a component which is imbedded in $\Vtorss^d\cup\cdots\cup\Vtorss^{r+1}$ or a component that does not intersect $\Vtors$,
we descart it (this won't have any effect on the veracity of our claim).
Next, let 
\[
Y^r=\Vtorss^r\cup W_1\cup\cdots\cup W_s
\]
be the decomposition of $Y^r$ such that
$\Vtorss^r$ is the union of all torsion cosets $B$ of dimension $r$ which are components of $Y^r$, and
$W_1,\ldots,W_s$ are the rest of irreducible components of $Y^r$.

For every $j=1,\ldots,s$\;, $W_j$ satisfies the following remark.
\begin{remark}
There does not exist any torsion coset $B$ such that $W_j\subseteq B\subseteq V$.
\end{remark}
\begin{proof}
If a torsion coset $B$ as such exists, $B\subset \Vtors$ and $\dim(B)\geq r$.
Therefore
\[
W_j\subseteq B\subseteq \Vtorss^d\cup\cdots\cup \Vtorss^r
\mbox{,}
\]
which contradicts the definition of $\Vtorss^r$ or the fact that no component of $Y^r$ is imbedded in $\Vtorss^d\cup\cdots\cup\Vtorss^{r+1}$.
\end{proof}

We apply Theorem \ref{t22} to $V_0=W_j$ and $V_1=V$, where $k_0=n-r\leq n-1$ and $k_1=n-d$.
Conclusion (a) of the theorem cannot be true, due to the previous remark;
hence there exists a hypersurface $Z_j$ of degree at most $\theta$ such that
$\overline{W_{j,\textrm{tors}}}\subset W_j\cap Z_j\varsubsetneq W_j$.
Krull's Hauptschatz implies that $W_j\cap Z_j$ is either empty or an $(r-1)$-equidimensional variety.
This allows us to define $Y^{r-1}$ as follows:
\[
Y^{r-1}=X^{r-1}\cup\displaystyle\bigcup_{j=1}^{s}(W_j\cap Z_j)
\mbox{.}
\]
By the construction of $Y^{r-1}$, assertion (i) of our claim is satisfied for $r-1$.

Moreover, by Bézout's theorem, the following inequality holds
\[
\deg(Y^{r-1})\leq 
\theta
\displaystyle\sum_{j=1}^s\deg(W_j)
+\deg(X^{r-1})
\mbox{.}
\]
Since $Y^r=\Vtorss^r\cup W_1\cup\cdots\cup W_s$, we have
\[
\deg(Y^{r-1})\leq
\theta\big( \deg(Y^r)-\deg(\Vtorss^r)\big)
+\deg(X^{r-1})
\mbox{.}
\]
Additioning $\sum_{i=r}^{d}\theta^{i+1-r}\deg(\Vtorss^i)$ to both sides of the inequality, we obtain
\[
\begin{array}{rcl}
\displaystyle\sum_{i=r}^{d}\theta^{i+1-r}\deg(\Vtorss^i)+\deg(Y^{r-1})
&\leq&
\displaystyle\sum_{i=r}^{d}\theta^{i+1-r}\deg(\Vtorss^i)
\\ \\
& &\;+\theta\big(\deg(Y^r)-\deg(\Vtorss^r)\big)+\deg(X^{r-1})
\\ \\
&=&
\theta\left(
\displaystyle\sum_{i=r+1}^{d}\theta^{i-r}\deg(\Vtorss^i)+\deg(Y^r)
\right)
\\ \\
& &\;+\deg(X^{r-1})
\mbox{.}
\end{array}
\]
By the induction,
we have $\sum_{i=r+1}^{d}\theta^{i-r}\deg(\Vtorss^i)+\deg(Y^r)\leq \sum_{i=r}^{d}\theta^{i-r}\deg(X^i)$.
Therefore
\[
\theta\left(
\displaystyle\sum_{i=r}^{d}\theta^{i-r}\deg(\Vtorss^i)+\deg(Y^r)
\right)
+\deg(X^{r-1})
\leq
\displaystyle\sum_{i=r-1}^{d}\theta^{i+1-r}\deg(X^i)
\mbox{,}
\]
proving that assertion (ii) of our claim holds for $r-1$.
\end{itemize}
\end{proof}

If $V$ is a hypersurface we can replace $\delta$ by $\deg(V)$.
Observe that this result is close to the conjectures.

\begin{remark}
Following the theorems presented by Amoroso and Viada \cite{AV09} we could obtain that $\delta_0(H)\leq \theta$,
for each maximal torsion coset $\bfomega H$ in $V$, $\delta_0(H)\leq \theta$.
However, we have the following sharper bound:
\[
\delta(H)\leq n\delta(V)
\mbox{,}
\]
which is a result of Bombieri and Gubler \cite[Theorem 3.3.8]{BG06}.
\end{remark}

\begin{remark}
In Theorem \ref{teo} we could give a more precise bound, depending on the field of definition of our variety $V$. 
To understand this, first observe that the varieties $V'$ we obtain in Proposition \ref{proposition21} are defined over the same field as~$V$.
Hence, in Theorem \ref{t21} we could consider changing the definition of $\theta_0$, depending on the field of definition of~$V$.

In the case that $V$ is not defined over $\Qab$, Theorem \ref{t21} remains true for
\[
\theta_0=2k(2d+1)\delta_0(V)
\mbox{.}
\]
Using this definition of $\theta_0$ in the induction theorems, we can improve the bound in Theorem \ref{teo} for this case.
Hence, if $V$ is not defined over $\Qab$, the number of maximal torsion cosets in $V$ is bounded by
\[
(2(2n-1)(n-1))^{n(n-k)}\delta(V)^n
\mbox{.}
\]

In the case that $V$ is defined over $\Qab$, this sharpening of the $\theta_0$ does not change significantly our bound since the order of $n$ in the constant would remain essentially the same.
\end{remark}

\section{Proof of the conjectures}
In this section we prove Theorem \ref{TEO}.
Observe that for any hypersurface $V$ given by the zeros of a polynomial $f$ of degree $\delta$,
Theorem \ref{teo} implies that the number of isolated torsion points on $V$ is bounded by $c_n \delta^n$.
A similar result with $\vol_n(\Delta)$ instead of $\delta^n$, where $\supp(f)\subset \Delta$, would imply the conjectures.
The idea to obtain such a result lies in considering another hypersurface $W$ with a degree depending only on $n$ and such that $\card(\Vtorss^0)\leq \vol_n(\Delta)\card (\Wtorss^0)$.
We are not able to give such a variety; instead, for positive integers $l\in\mathbb{Z}$, we build a hypersurface $W$ of degree $2nl+\kappa_{n,\Delta}$ and such that $\card(\Vtorss^0)\leq \tilde{\kappa}_{n} l^{-n}\vol_n(\Delta)\card(\Wtorss^0)$, where $\kappa_{n,\Delta}$ depends only on $n$ and $\Delta$, and $\tilde{\kappa}_{n}$ on $n$.
Taking the limit for $l\rightarrow\infty$ we obtain the statement we expect.  

First, we state a result of John \cite[Theorem III]{Joh48} which allows us to compare the volume of any convex polytope $\Delta$ with the volume of the ellipsoid of smallest volume containing $\Delta$.
\begin{theorem}\label{John}
If $E$ is the ellipsoid of smallest volume containing a set $S$ in $\mathbb{R}^n$, then the ellipsoid $\frac{1}{n}E$ is contained in the convex hull of $S$, $\conv (S)$.
\end{theorem}

An ellipsoid $E$ is determined by an invertible matrix $M\in\GL_n(\mathbb{R})$ and a vector $\bfv\in\mathbb{R}^n$ such that
\[
E=M B_n+\bfv
\mbox{,}
\]
where $B_n$ is the $n$-dimensional unit ball with center $\mathbf{0}$.
In particular, the volume of $E$ is detemined by $M$ since $\vol_n(E)=\det(M)\omega_n$,
where $\omega_n$ is the $n$-volume of $n$-sphere.

In the following theorem we consider a convex polytope $\Delta$ with integer vertices.
By means of John's result, we include a deformation of $\Delta$ in a homothety of the standard simplex $\mathcal{S}_n=\lbrace \bft\in (\mathbb{R}_{\geq 0})^n \mid\; t_1+\cdots +t_n\leq 1\rbrace$ with comparable volume.

\begin{proposition}\label{pconvex}
Let $\Delta$ be a convex polytope with integer vertexes.
For any $l\in\mathbb{N}_{>0}$, there exists a non-singular integer matrix $M_l$ and an integer vector $\bftau_l$ such that
\[
M_l\Delta+\bftau_l
\subset
2n(l+n\diam_1(\Delta)+n)\mathcal{S}_n
\mbox{,}
\]
and
\begin{equation}\label{ineq}
\lim\limits_{l\rightarrow +\infty}l^{-n}\det(M_l)\geq \omega_n n^{-n}\vol_n(\Delta)^{-1}
\mbox{.}
\end{equation}
\end{proposition}
\begin{proof}
Translating $\Delta$ by an integer vector,
we can always assume that $\Delta$ lies in $(\mathbb{R}_{\geq 0})^n$ and that $\Delta\cap \lbrace \bfx\in\mathbb{Z}^n \mid \bfx_i=0\rbrace\neq \emptyset$, for every $i=1,\ldots,n$.
Thus for any matrix $N\in\mathcal{M}_{n\times n}(\mathbb{R})$ with maximum norm $\norm{N}\leq 1$, we have $N\,\Delta\subset n\diam_1(\Delta)B_n$.

Let $E$ be the ellipsoid of smallest volume containing $\Delta$.
Let $M\in\GL(\mathbb{R})$ and $\bfv\in\mathbb{R}^n$ be the corresponding matrix and vector such that $M\,B_n+\bfv=E$.
In particular, we have that $B_n$ is the ellipsoid of smallest volume containing $M^{-1}\Delta-\bfv$.

Next, take $M_l\in\GL(\mathbb{Z})$ and $\bfv_l\in\mathbb{Z}^n$ to be integer approximations of $lM^{-1}$ and $l\bfv$ respectively; that is
\[
\begin{array}{l}
lM^{-1}=M_l+M',\quad \norm{M'}<1;\\
\\
l\bfv=\bfv_l+\bfv',\quad \norm{\bfv'}<1;
\end{array}
\]
where $\norm{\cdot}$ denotes the maximum norm.
Take
\[
\bftau_l=(l+n\diam_{1}(\Delta)+n)\mathbf{1}-\bfv_l
\mbox{.}
\]

We proceed to bound the domain where $M_l\Delta+\bftau_l$ lies. To do that, we develop as follows:
\[
M_l\Delta+\bftau_l=l(M^{-1}\Delta -\bfv)-M'\Delta+\bfv'+(l+n\diam_{1}(\Delta)+n)\mathbf{1}
\mbox{.}
\]
We have that $l(M^{-1}\Delta -\bfv)\subset lB_n$ and $\bfv'\in nB_n$.
Moreover, since $\norm{M'}\leq 1$, $M'\Delta \subset n\diam_1(\Delta)B_n$.
Putting it all together we obtain
\begin{multline*}
M_l\Delta+\bftau_l
\subset (l+n\diam_1(\Delta)+n)B_n+(l+n\diam_1(\Delta)+n)\mathbf{1}\\
\subset 2n(l+n\diam_1(\Delta)+n)\mathcal{S}_n
\mbox{.}
\end{multline*}

It is left to prove (\ref{ineq}).
Using John's result (Theorem \ref{John}), we have that $E\subset n\Delta$.
Therefore
\[
\vol_n(E)=\omega_n\det(M)\leq n^n\vol_n(\Delta)\mbox{.}
\]
In addition, by our choice of $M_l$, we have that
\[
\lim\limits_{l\rightarrow +\infty}l^{-n}\det(M_l)=\det(M)^{-1}
\mbox{.}
\]
Inequality (\ref{ineq}) follows directly from here.

\end{proof}

This proposition allows us to extend the result in Theorem \ref{teo}.
First, we introduce a different notion of degree.
Let $\Delta\subset\mathbb{R}^n$ be a convex polytope with integral vertexes
and let $\psi_\Delta:\Gm^n\rightarrow\mathbb{P}^{N-1}$, with $N=\card(\Delta\cap\mathbb{Z}^n)$,
be the morphism mapping $\bft\mapsto(\bft^{\bfa})_{\bfa\in\Delta\cap\mathbb{Z}^n}$.
For any variety $V\subset\Gm^n$, we define the $\Delta$-degree of $V$ as
\[
\deg_{\Delta}(V)=\deg (\psi_\Delta(V))
\mbox{.}
\]
In particular, if $V$ is of dimension $d$
we can find general polynomials $f_1,\ldots,f_d$ with support $\supp(f)\subset\Delta$ and 
\[
\deg_{\Delta}(V)=\card(V\cap Z(f_1,\ldots, f_d))
\mbox{,}
\]
where $Z(f_1,\ldots, f_d)$ is the subvariety of $\Gm^n$ defined by $f_1,\ldots,f_n$.

\vspace*{2mm}

The following result is a general statement which easily implies Ruppert's and Aliev-Smyth's conjectures.

\begin{theorem14}
Let $V\subset\Gm^n$ be a variety of dimension $d$, defined by polynomials with support in the convex polytope $\Delta$.
For $j=0,\ldots,d$\, let $\Vtorss^j$ be the $j$-equidimensional part of $\Vtors$.
Then
\[
\deg_{\Delta}(\Vtorss^j)\leq \widetilde{c}_{n,j}\vol_n(\Delta)
\]
for every $j=0,\ldots,d$\,, where
\[
\widetilde{c}_{n,j}=((2n-1)(n-1)(2^{2n}+2^{n+1}-2))^{(n-1)(n-j)} 2^n n^{2n}\omega_n^{-1}
\mbox{,}
\]
and $\omega_n$ is the volume of the $n$-sphere.
\end{theorem14}

\begin{proof}
Let $M_l$ and $\bftau_l$ be as in Proposition \ref{pconvex}.
Let $\varphi:\Gm^n\rightarrow\Gm^n$ be the endomorphism defined by $M_l$,
mapping $\bfx\mapsto \bfx^{M_l}$.
By Proposition \ref{pconvex}, for any polynomial $f$ with support $\supp(f)\subset\Delta$, we have
\[
\supp(f(\bfx^{M_l})\cdot\bfx^{\bftau_l})\subset 2n(l+n\diam_1(\Delta)+n)\mathcal{S}_n
\mbox{.}
\]

Let $W=\varphi^{-1}(V)$. Since $V$ is defined by polynomials supported in $\Delta$,
$W$ can be defined by polynomials of degree at most $2n(l+\diam_1(\Delta)+n)$.
Morover, for every $j=0,\ldots,d$\,, we have that $\varphi^{-1}(\Vtorss^j)=\Wtorss^j$.

Fix $j$. By Theorem \ref{teo} we have the following inequality:
\begin{equation}\label{eqTEO1}
\deg(\Wtorss^j)\leq c_{n,j} 2n(l+n\diam_1(\Delta)+n)^{n-j}
\mbox{.}
\end{equation}

We proceed to compare $\deg(\Wtorss^j)$ and $\deg_\Delta(\Vtorss^j)$.
To do this, take $f_1,\ldots,f_j$ generic polynomials such that $\supp(f_i)\subset\Delta$ and
\[
\deg_\Delta(\Vtorss^j)=\card(\Vtorss^j\cap Z(f_1,\ldots ,f_j))
\mbox{.}
\]
Since $f\circ \varphi^{-1}= f(\bfx^{M_l})$ and $f(\bfx^{M_l})\cdot\bfx^{\bftau_l}$ define the same variety,
\[
\varphi^{-1}(\Vtorss^j\cap Z(f_1,\ldots ,f_j))=
\Wtorss^j\cap Z(f_1(\bfx^{M_l})\cdot\bfx^{\bftau_l},\ldots,f_j(\bfx^{M_l})\cdot\bfx^{\bftau_l})
\mbox{.}
\]
Since $f_1,\ldots,f_j$ are generic, by Bézout we have that
\begin{multline*}
\card(\Wtorss^j\cap Z(f_1(\bfx^{M_l})\cdot\bfx^{\bftau_l},\ldots ,f_j(\bfx^{M_l})\cdot\bfx^{\bftau_l}))
\\ \leq \deg(\Wtorss^j)2n(l+n\diam_1(\Delta)+n)^j
\mbox{.}
\end{multline*}
Since $\card(\varphi^{-1}(\bfx))=\det(M_l)$ for any point $\bfx\in\Gm^n$, we have
\begin{multline*}
\det(M_l)\deg_\Delta(\Vtorss^j)
=\card(\varphi^{-1}(\Vtorss^j\cap Z(f_1, \ldots ,f_j)))
\\ \leq \deg(\Wtorss^j)2n(l+n\diam_1(\Delta)+n)^j
\mbox{.}
\end{multline*}
Combining this inequality with (\ref{eqTEO1}), we obtain
\begin{equation}\label{eqTEO2}
\deg_\Delta(\Vtorss^j)\leq c_{n,j}(2n(l+n\diam_1(\Delta)+n))^{n}\det(M_l)^{-1}
\mbox{.}
\end{equation}
By Proposition \ref{pconvex}
\[
\lim\limits_{l\rightarrow +\infty}l^n\det(M_l)^{-1}\leq n^n\omega_n^{-1}\vol_n(\Delta)
\mbox{.}
\]
Hence, taking the limit for $l\rightarrow +\infty$ in (\ref{eqTEO2}), we get
\[
\deg_\Delta(\Vtorss^j)\leq
c_{n,j} 2^n n^{2n}\omega_n^{-1}\vol_n(\Delta)
\mbox{.}
\]
\end{proof}

Let $V\subset\Gm^n$ be a hypersurface given by a polynomial $f\in\overline{\mathbb{Q}}[X_1,\ldots,X_n]$. If we take $\Delta=[0,d_1]\times\cdots\times[0,d_n]$ where $(d_1,\ldots,d_n)$ is the multidegree of $f$, Theorem \ref{TEO} for $j=0$ proves Ruppert's conjecture (Conjecture \ref{ConjR}).
A slightly better result can be obtained applying Theorem \ref{teo} directly to the hypersurface $W$
defined by $f(x_1^{D_1},\ldots,x_n^{D_n})=0$, with $D_i=d_1\cdots d_n/d_i$.
In this case we obtain that
\[
\Vtorss^0\leq n^n c_{n,0} d_1\cdots d_n
\mbox{.}
\]

On the other hand, if we take $\Delta=\conv(\supp(f))$, Theorem \ref{TEO} proves Aliev and Smyth's conjecture (Conjecture \ref{ConjAS}).

Moreover, by comparing the bound in Theorem \ref{teo} and the bound in Theorem \ref{TEO} for the dense case
($\conv(\supp(f))=\deg(f)\mathcal{S}_n$),
we can observe that they differ only by a multiplying factor $2^n n^{2n}\omega_n$ which does not increase the order of the constant.

\section{Example}

We build an example to show that
the exponent of $d$ in Theorem \ref{teo} is optimal and the constant $c_n$ must depend on $n$.
To do this, we need first a result of Conway and Jones on vanishing sums of roots of unity.
Let us define, for $m\in\mathbb{N}_{>0}$,
\[
\Psi(m):=2+\sum_{
\begin{smallmatrix}
p\mid m\\
p\mbox{\scriptsize{ prime}}
\end{smallmatrix}}
(p-2)
\mbox{.}
\]

The result of Conway and Jones is the following.
\begin{theorem}[\cite{CJ76}]\label{CJ}
Let $\xi_1,\ldots,\xi_N$ be $N$ roots of unity.
Let $a_1,\ldots,a_N\in\mathbb{Z}$ such that
$S=a_1\xi_1+\ldots +a_N\xi_N=0$
is minimal (i.e.\ there are no non-trivial vanishing subsums of $S$).
Let 
\[
m=\lcm (\ord(\xi_2/\xi_1),\ldots,\ord(\xi_N/\xi_1))
\mbox{.}
\]
Then $m$ is squarefree and $\Psi(m)\leq N$.
\end{theorem}

\begin{example}
First of all, let $p_1,\ldots,p_n$ be $n$ different primes such that $p_i>2n$ for all $i=1,\ldots,n$.
In particular, we will have that $\Psi(p_{j} p_{k})> 2n$ for every different $j$ and $k$.
Let $W$ be the variety defined by
\[
g(X_1,\ldots,X_n)= \zeta_{p_1}+\cdots+\zeta_{p_n}+X_1+\ldots+X_n=0
\mbox{.}
\]
We claim that if $\bfomega\in \Wtorss$, then
\begin{equation*}\label{condition}
\lbrace\omega_1,\ldots,\omega_n\rbrace=\lbrace -\zeta_{p_1},\ldots,-\zeta_{p_n}\rbrace
\mbox{.}
\end{equation*}

\begin{proof}
Take $\bfomega\in\mu_\infty^n$ such that $g(\bfomega)=0$
and consider
\[
S=g(\bfomega)=\zeta_{p_1}+\cdots+\zeta_{p_n}+\omega_1+\cdots+\omega_n
\mbox{.}
\]

Let $S=S_1+\cdots +S_t$ be a decomposition of $S$ in minimal subsums,
such that $S_i=0$ for every $i=1,\ldots, r$.
If, up to reordering, $t=n$ and $S_i=\zeta_{p_i}+\omega_i$, we are done.

Suppose that this is not the case.
Hence, there exists a minimal vanishing non-trivial subsum $S'$ with at least three elements.
Without loss of generality, we can assume that $S'$ has $\zeta_{p_j}$ and $\zeta_{p_k}$ as summands, for some different $j$ and $k$.
We take $m'$ to be the equivalent of $m$ in Theorem \ref{CJ} with respect to $S'$, we have that $p_j p_k\mid m'$, so $\Psi(m')\geq \Psi(p_j p_k)>2n$.
On the other hand, since $S'$ is a minimal sum with less than $2n$ summands,
Theorem \ref{CJ} states $\Psi(m')< 2n$.
This contradicts the fact that $\Psi(m')>2n$.
Therefore, there is no vanishing subsum of $S$ and our claim is proved.
\end{proof}

Since our claim holds, we have
\[
\Wtorss=\left\lbrace
\bfomega\in\Gm^n \mid \;
\lbrace\omega_1,\ldots,\omega_n\rbrace=\lbrace -\zeta_{p_1},\ldots,-\zeta_{p_n}\rbrace
\right\rbrace
\mbox{.}
\]
So $\Wtors=\Wtorss$ is a discrete ensemble with $n!$ elements.

It is enough to take $V=[d]^{-1}(W)$, which is the hypersurface in $\Gm^n$ defined by
\[
f(X_1,\ldots,X_n)=\zeta_{p_1}+\cdots+\zeta_{p_n}+X_1^d+\ldots+X_n^d=0
\mbox{.}
\]
Then, we have that $\Vtors=[d]^{-1}(\Wtors)$ which is the following discrete ensemble:
\[
\Vtors=\left\lbrace
\bfomega\in\Gm^n\mid\;
\lbrace \omega_1^d,\ldots,\omega_n^d\rbrace
=
\lbrace -\zeta_{p_1},\ldots,-\zeta_{p_n}\rbrace
\right\rbrace
\mbox{.}
\]
Hence, the number of isolated torsion points in $V$ is $n!\, d^n$.
In this case, this is the number of maximal torsion cosets in $V$.
\end{example}

By considering the homomorphism $[d_1,\ldots,d_n]:\Gm^n\rightarrow\Gm^n$, mapping $(t_1,\ldots,t_n)\mapsto(t_1^{d_1},\ldots,t_n^{d_n})$, instead of simply $[d]$, we would obtain the variety $V$ defined by
\[
f(x_1,\ldots,x_n)=\zeta_{p_1}+\cdots+\zeta_{p_n}+X_1^{d_1}+\ldots+X_n^{d_n}=0
\mbox{.}
\]
In this case, the number of isolated torsion points in $V$ is $n!\,d_1\cdots d_n$ which shows the effectiveness of the bound in terms of the multidegree.

\bibliographystyle{amsalpha}
\bibliography{biblio}

\end{document}